\documentclass[11pt,reqno,oneside]{amsart}
\usepackage{amscd}
\usepackage{amsmath}
\usepackage{latexsym}
\usepackage{amsfonts}
\usepackage{amssymb}
\usepackage{amsthm}
\usepackage{graphicx}
\usepackage{verbatim}
\usepackage{mathrsfs}
\usepackage{enumerate}
\usepackage{hyperref}
\usepackage{fullpage}
\usepackage{xcolor}

\theoremstyle{plain}
\theoremstyle{plain}\newtheorem{theorem}{Theorem}[section]
\theoremstyle{plain}\newtheorem{lemma}[theorem]{Lemma}
\theoremstyle{plain}
\theoremstyle{plain}
\theoremstyle{plain}

\newcommand{\R}{\mathbb{R}}
\newcommand{\be}{\begin{equation}}
\newcommand{\ee}{\end{equation}}
 \newcommand{\ba}{\begin{aligned}}
 \newcommand{\ea}{\end{aligned}}

  \newcommand{\ben}{\begin{enumerate}}
   \newcommand{\een}{\end{enumerate}}

\newcommand{\Rmnum}[1]{\expandafter\@slowromancap\romannumeral #1@}
\allowdisplaybreaks

\begin{document}
%%%%%%%%%%%%%%%%%%%%%%%%%%%%%%%%%%%%%%%%%%%%%%%%%%%%%%%%%%%%%%%%%%%%%%%%%%%%%%%%%%%%%%%%%%%%%%%%%%%%
\title{An Extension of Calder$\acute{\rm O}$n-Zygmund type singular integral}

\author[Q. Jiu, D. Li, H. Yu]{ Quansen Jiu$^{1}$, Dongsheng Li$^{2}$, Huan Yu$^{3}$,}

\address{$^1$ School of Mathematical Sciences, Capital Normal University, Beijing, 100048, P.R.China}

\email{jiuqs@cnu.edu.cn}

\address{$^2$  College of Science, Xi'an Jiaotong University, Xi'an 710049, P.R.China}

\email{lidsh@mail.xjtu.edu.cn}

\address{$^3$  School of Applied Science, Beijing Information Science and Technology University, Beijing, 100192, P.R.China}

\email{huanyu@bistu.edu.cn}

\date{\today}

\begin{abstract}

In this paper,  we  consider  a kind of singular integral which can be viewed as an extension of the classical Calder$\acute{\rm o}$n-Zygmund type singular integral. We establish an estimate of the  singular integral in the  $L^q$ space for $1<q<\infty$. In particular, the  Calder$\acute{\rm o}$n-Zygmund estimate 
 can be recovered from our obtained estimate. The proof of our main result is via the so called "geometric approach", which was applied in \cite{CP} on the $L^q$ estimate of the elliptic equations and in \cite{LW,Wang} on a new proof of the the   Calder$\acute{\rm o}$n-Zygmund estimate.

\end{abstract}
\smallskip
\maketitle%\centerline{$^{1}$Institute of Applied Physics and Computational Mathematics}
%\centerline{Beijing  100088, P. R. China}
%\centerline{\it Email: chen\_qionglei@iapcm.ac.cn}

%\centerline{\scshape Huan Yu}
%\centerline{$^{2}$Institute of Applied Physics and Computational Mathematics}
%\centerline{Beijing  100088, P. R. China}
%\centerline{\it Email: yuhuandreamer@163.com}

%%%%%%%%%%

\section{Introduction and Main Results}
 The classical Calder$\acute{\rm o}$n-Zygmund type singular integral is defined as
\begin{equation}\label{4-9-17}
I_\varepsilon(f)(x)=\int_{|y|\ge\varepsilon}\frac{\Omega(y)}{|y|^n}f(x-y) dy
\end{equation}
for any $\varepsilon>0$ and $f\in L^q(\mathbb{R}^n), 1\le q<\infty$, where the function
 $\Omega: \mathbb{R}^n \rightarrow  \mathbb{R}$  is bounded and homogeneous of degree $0$, and satisfies  the cancellation property and "Dini-type" smoothness property. More precisely, there exist two positive constants $B_1$ and $B_2$ such that for any $r>0$ and $x\in  \mathbb{R}^n\setminus \{0\}$, $\Omega$ satisfies
\begin{equation}\label{4-9-15}
\left\{
\begin{array}{lll}
 & |\Omega(x)|\le B_1, \ &(\rm bounded) \\[3mm]
 &\Omega(rx)=\Omega(x),\ & (\rm homogeneous\  of\  degree \ 0)\\[3mm]
&\int_{\mathbb{S}^1} \Omega(x) d\sigma =0, \ &(\rm cancellation)\\[3mm]
& \int_0^1\frac{\omega(\delta)}{\delta} d\delta=B_2,  \ &(\rm Dini\ type \ continuous)
\end{array}
\right.
\end{equation}
where $\mathbb{S}^1=\{x\in \mathbb{R}^n: |x|=1\}$ is the unit sphere in $\mathbb{R}^n$ and $\omega$ is defined as
\begin{equation}\label{4-9-16}
\omega(\delta)=\sup\{|\Omega(x)-\Omega(x')|: |x-x'|\le \delta, |x|=|x'|=1\}.
\end{equation}
It is clear that Riesz transform is a specific example of the Calder$\acute{\rm o}$n-Zygmund singular integral if we take $\Omega(x)=\frac{x_j}{|x|}$ for $j=1,2,\cdots,n$ respectively.

Under assumptions \eqref{4-9-15},  the well-known Calder$\acute{\rm o}$n-Zygmund  estimate   reads as (see \cite{CZ, stein})
\begin{theorem}\label{theorem1}
Suppose $1<q<\infty$. There exists a constant $A$ depending only on $n, q, B_1$ and $B_2$ such that $\|I_\varepsilon f\|_q\le A\|f\|_q$ for any $\varepsilon>0$ and $f\in L^q(\mathbb{R}^n)$.
\end{theorem}

The original proof of Theorem  \ref{theorem1} is classical and can be found in \cite{CZ, stein}. Motivated by the proof in \cite{CP} on $L^q$ estimates of elliptic equations,  Li and Wang \cite{LW} presented a new proof of Theorem \ref{theorem1}, which is so called "geometric approach" (see \cite{Wang}). One of new  ingredients in the proof  in \cite{LW} lies in that it is proved directly that $I_\varepsilon$ is of strong type $(q,q)$ for $2<q<\infty$, based on the strong type (2,2) estimate of $I_\varepsilon$, which is different from the  original  one for $1<q<2$ in \cite{CZ, stein}.

In this paper,  we consider the following singular integral
\begin{equation}\label{T-operator}
T_\varepsilon(f)(x)=\int_{|y|\ge\varepsilon}\frac{\Omega(y)}{|y|^{n-\beta}}f(x-y) dy
\end{equation}
for any $\varepsilon>0$ and $f\in  L^1(\mathbb{R}^n)\cap L^q(\mathbb{R}^n), 1\le q<\infty$ and $0< \beta<n$, where $\Omega$ is same as in \eqref{4-9-17}. Formally,  $T_\varepsilon$ becomes the Calder$\acute{\rm o}$n-Zygmund type singular integral $I_\varepsilon$ when $\beta=0$.

The main aim of this paper is to obtain a uniform $L^q (1<q<\infty)$ estimate of $T_\varepsilon$ with respect to $\beta>0$ such that the strong type $(q,q)$ estimate to the  Calder$\acute{\rm o}$n-Zygmund type singular integral $I_\varepsilon$ can be recovered when $\beta\to 0$. Our main result can be stated as
\begin{theorem}\label{theorem2}
Let $0<\beta_0<\frac 12$ be any fixed and small number. Then for any $f\in L^1(\mathbb{R}^n)\cap L^q(\mathbb{R}^n)$ with $1< q <\infty$, there exists an absolute constant $C$ depending on $n, q, B_1, B_2$ and $\beta_0$ such that
\begin{equation}\label{4-11-10}
\begin{split}
\|T_\varepsilon f\|_q\le C\big(\|f\|_q+\frac{\beta^{\frac{(q-1)n}{q}}}{\sqrt[q]{(n(q-1)-\beta q)}}\|f\|_1\big)
\end{split}
\end{equation}
holds  uniformly for $\varepsilon>0$ and  $0<\beta<\min\{1-\beta_0,\frac{(q-1)n}{q}\}$.
\end{theorem}
  In view of  Riesz potential (see \cite{stein}), it is direct to obtain that for any  $0<\beta<n$,
 \begin{equation}\label{5-11-1}
\begin{split}
\|T_\varepsilon f\|_q \le C(n, q, \beta)\|f\|_p, \ \ \frac 1q=\frac 1p-\frac{\beta}{n},
\end{split}
\end{equation}
where the constant $C(n,q,\beta)$ depends on $n, q$ and $\beta$ and the estimate \eqref{4-11-10} can be obtained by interpolation for $0<\beta<n$. However,  the constant $C(n,q,\beta)$  will become unbounded when $\beta\to 0$. It should be addressed  that the constant $C$ on the right hand of \eqref{4-11-10}  does not depend on $\beta$ and   the strong $(q,q)$ type estimate of the Calder$\acute{\rm o}$n-Zygmund type singular integral $I_\varepsilon$ can be recovered from  \eqref{4-11-10} when $\beta\to 0$. In this sense, the singular integral $T_\varepsilon$ in \eqref{T-operator} can be viewed as an extension of the Calder$\acute{\rm o}$n-Zygmund type singular integral. In particular, when taking $\Omega(x)=\frac{x_j}{|x|}$ for $j=1,2,\cdots,n$ respectively,  Yu and Jiu   \cite{YJ} proved
\begin{theorem}\label{theorem3}
Take $\Omega(x)=\frac{x_j}{|x|}$ for $j=1,2,\cdots,n$ in \eqref{T-operator} respectively. Let $0<\beta_0<\frac 12$ be any fixed and small number. Then for any $f\in L^1(\mathbb{R}^n)\cap L^q(\mathbb{R}^n)$ with $1< q <\infty$, there exists an absolute constant $C$ depending on $n, q$ and $\beta_0$ such that

(1) if $q=2$,  there holds
\begin{equation}\label{4-11-2}
\begin{split}
\|T_\varepsilon f\|_2\le C\big(\|f\|_2+\frac{\beta^{\frac n2}}{\sqrt{n-2\beta}}\|f\|_1\big)
\end{split}
\end{equation}
 for $0<\beta<1-\beta_0$;

(2) if $1<q<2$, there holds
\begin{equation}\label{3275}
\begin{split}
\|T_\varepsilon f\|_q\le C\big(\frac{1}{\sqrt[q]{(q-1)(2-q)}}\|f\|_q+\frac{1}{\sqrt[q]{q(n-\beta)^2-n^2}}\|f\|_p+\frac{\beta^{\frac{n(q-1)}{q}}}{\sqrt[q]{(n(q-1)-\beta q)}}\|f\|_1\big),
\end{split}
\end{equation}
where  $\frac{1}{q}=\frac{1}{p}(1-\frac{\beta}{n})$, $0<\beta<\min\{1-\beta_0,\frac{(q-\sqrt{q})n}{q}\}$;

(3) if $2<q<\infty$,  there holds
\begin{equation}\label{3276}
\begin{split}
\|T_\varepsilon f\|_{(L^{q'}\cap L^{p'})^*}
\le & C\max\{\frac{1}{\sqrt[q']{(q'-1)(2-q')}},\frac{1}{\sqrt[q']{q'(n-\beta)^2-n^2}}\} \|f\|_q\\
&+C\frac{\beta^{\frac{n(q-1)}{q}}}{\sqrt[q]{(n(q-1)-\beta q)}}\|f\|_1,
\end{split}
\end{equation}
where $0<\beta<\min\{1-\beta_0,\frac{(q'-\sqrt{q'})n}{q'}\}$,  $\frac 1q+\frac{1}{q'}=1$, $\frac{1}{q'}=\frac{1}{p'}(1-\frac{\beta}{n})$, and $(L^{q'}\cap L^{p'})^*$ is the dual space of $L^{q'}\cap L^{p'}$.
\end{theorem}
It follows from Theorem \ref{theorem3} that the strong $(q,q)$ type estimate of the Riesz transform can be recovered as $\beta\to 0$ and hence the singular integral $T_\varepsilon$ can be viewed as an extension of the Riesz transform if we take $\Omega(x)=\frac{x_j}{|x|}$ for $j=1,2,\cdots,n$ in \eqref{T-operator}, respectively. This kind of singular integral appears in  the approximation of the surface quasi-geostrophic equation (SQG) equation from the generalized SQG equation, see \cite{YJ, YZJ} for more details.  However,  the estimate \eqref{3275}  on $\|T_\varepsilon f\|_q$ is in terms of $\|f\|_q, \|f\|_p$ and $\|f\|_1$ and the estimate \eqref{3276} is on $\|T_\varepsilon f\|_{(L^{q'}\cap L^{p'})^*}$ but not on $\|T_\varepsilon f\|_q$, although the right hand side  is in terms of $\|f\|_q$ and $\|f\|_1$. As mentioned  in \cite{YJ} (see Remark 1.3 in \cite{YJ}), it would be interesting to make an estimate of $T_\varepsilon f$ in $L^q$ space for $2<q<\infty$ in a direct way. In this paper, we completely answer this question and  obtain the  estimate $\|T_\varepsilon f\|_q$ in terms of  $\|f\|_q$ and $\|f\|_1$ for $1<q<\infty$.  Moreover, the singular integral $T_\varepsilon$ considered in Theorem \ref{theorem2} is much more general than that in Theorem \ref{theorem3}.

 The approaches between \cite {YJ} and this paper are different. The main approach in \cite{YJ} is similar to the original one in \cite{CZ,stein} in which Theorem \ref{theorem1} was proved.  The main approach of this paper is similar to the one in \cite{LW, Wang} via the so called "geometric approach". As in \cite{YJ,YZJ}, we split the singular integral \eqref{T-operator} into two parts: the part near the origin denoted by $T_1f$ and the one apart from the origin denoted by $T_2f$. The estimate on $\|T_2f\|_q$  is easy to obtain (see Lemma \ref{uni-est1}) and the key part is to estimate $\|T_1f\|_q$ for $1<q<\infty$ (see Lemma \ref{theorem21}). We will first estimate $\|T_1f\|_2$ by modifying  the proof in \cite{YZJ,YJ} in Lemma \ref{uni-est21}. Then we will estimate $\|T_1f\|_q$ for $2<q<\infty$. To this end,  one of our main contributions  is to prove Lemma \ref{MainLemma}, which is a cornerstone to prove the main result. We should  remark that in Theorems \ref{theorem2}-\ref{theorem3} the restriction on $0<\beta\le 1-\beta_0$ for any small $0<\beta_0<\frac12$ (or $0<\beta<1$) is from the the estimate $\|T_1f\|_2$ in Lemma \ref{uni-est21}   (see \eqref{I-E} for more details) and $0<\beta<\frac{(q-1)n}{q}$ is from the estimate on $\|T_2f\|_q$ for $1<q<\infty$.

The paper is organized as follows. In Section 2, we will present some preliminary estimates which will be needed later. The proof of  Theorem \ref{theorem2}  will be given in Sections 3.

{\it Notations.} For $f\in L^q(\mathbb{R}^n)$ with  $1\le q\le \infty$, we denote its $L^q$ norm by $\|f\|_q$  and its support set $\bar{\{x\in \mathbb{R}^n: f(x)\neq 0\}}$ by supp $f$. For any cube $Q$ in $\mathbb{R}^n$ and $a>0$, we denote $aQ$ the cube with same center as $Q$ and the side-length $al$, where $l$ is the side-length of $Q$. Throughout this paper, all the cubes are open. For a measurable set $E$ in $\mathbb{R}^n$, $|E|$ means its Lebesgue measure. For a measurebale function $f$, its Hardy-Littlewood maximal function mathcal{M}f is defined as
$$
\displaystyle{\mathcal{M}f(x) =\sup_{x\in Q}}\frac{1}{|Q|}\int_Q |f(y)| dy
$$
for any cube $Q\subset \mathbb{R}^n$ containing $x$.

By $A\sim B$, we mean that $A$ is equivalent to $B$, that is, there exist two positive constants $c$ and $C$ such that $cA\le B\le CA$.

\section{Preliminaries}
\setcounter{section}{2}\setcounter{equation}{0}

In this section, we first state that the maximal function $\mathcal{M}f$  is  of strong type $(q,q)$ for $1<q<\infty$ and of weak type $(1,1)$, which is
\begin{lemma}\label{MF} It holds that
 \begin{equation}\label{MF1}
\begin{split}
&\|\mathcal{M}f\|_q\le C_q\|f\|_q,\\
&|\{x\in \mathbb{R}^n: (\mathcal{M}f)(x)>\lambda\}|\le C_1\frac{\|f\|_1}{\lambda}
\end{split}
\end{equation}
for any $\lambda>0$, where $C_1$ is a positive constant depending only on $n$ and $C_q$ is a positive constant depending only on $n$ and $q$ for each $1<q<\infty$.
\end{lemma}

Then we split the singular integral \eqref{T-operator} into two parts: the one near the origin and the one apart from the origin. Let  $\chi(s)\in C_0^\infty(R)$ be the usual smooth cutting-off function which is defined as
$$
\chi(s)=
\left\{
\begin{array}{ll}
1, & |s|\le 1,\\[3mm]
0, & |s|\ge 2,
\end{array}
\right.
$$
satisfying $|\chi'(s)|\le 2$. Let
\begin{equation}\label{Chi-L}
\chi_\lambda(s)=\chi(\lambda s),
\end{equation}
and define
\begin{eqnarray}
&& T_1f(x)=\int_{|y|\ge\varepsilon}\frac{\Omega(y)}{|y|^{n-\beta}}\chi_\beta(|y|)f(x-y) dy \label{T11}\\
&& T_2f(x)=\int_{|y|\ge\varepsilon}\frac{\Omega(y)}{|y|^{n-\beta}}(1-\chi_\beta(|y|))f(x-y) \label{T22} dy.
\end{eqnarray}
 It is clear that the operator $T$ in \eqref{T-operator} can be written as
\begin{equation}\label{T-j}
T=T_1+T_2.
\end{equation}

The following is a $L^q$-estimate of $T_2$:
\begin{lemma}\label{uni-est1}
 There exists an absolute constant $C$ depending on $B_1$ and independent of $\beta$ such that for any $1<q<\infty$,
\begin{equation}\label{T2-11}
\|T_2f\|_q\leq C\frac{\beta^{\frac{(q-1)n}{q}}}{\sqrt[q]{(n(q-1)-\beta q)}}\|f\|_1,  0<\beta<\frac{(q-1)n}{q}.
\end{equation}
\end{lemma}
\begin{proof}%[Proof of Lemma \ref{uni-est1}]
Thanks to \eqref{4-9-15}, direct estimates give
\begin{equation*}\begin{split}
\|T_2f\|_q&\leq\|\int_{|x-y|\geq\frac{1}{\beta}}\frac{1}{|x-y|^{n-\beta}}|f(y)|\,dy\|_q
\\&\leq\|f\|_1(\int_{|x-y|\geq\frac{1}{\beta}}\frac{1}{|x-y|^{q(n-\beta)}}\,dy)^{\frac1q}
\\&\leq C\frac{\beta^{\frac{(q-1)n}{q}}}{\sqrt[q]{(n(q-1)-\beta q)}}\|f\|_1
\end{split}
\end{equation*}
for any $0<\beta<\frac{(q-1)n}{q}$.
\end{proof}

In view of Lemma \ref{uni-est1}, to prove Theorem \ref{theorem2}, we only need to prove
\begin{theorem}\label{theorem21}
Let $0<\beta_0<\frac 12$ be any fixed and small number. Then for any $f\in L^1(\mathbb{R}^n)\cap L^q(\mathbb{R}^n)$ with $1< q <\infty$, then there exists an absolute constant $C$ depending on $n, q, B_1, B_2$ and $\beta_0$ such that
\begin{equation}\label{4-11-1}
\begin{split}
\|T_1f\|_q\le C\|f\|_q
\end{split}
\end{equation}
holds for $0<\beta\le 1-\beta_0$.
\end{theorem}
We first prove that the operator $T_1$ is of strong type $(2,2)$, which has been shown in \cite{YZJ,YJ} for the case $\Omega(x)=\frac{x_j}{|x|}$ for $j=1,2,\cdots,n$% in \eqref{T-operator}, respectively
. Here we modify the proof in \cite{YZJ,YJ} to obtain
\begin{lemma}\label{uni-est21}
 Let $0<\beta_0<\frac 12$ be any fixed and small number. Then there exists an absolute constant $C$ depending on $n$ and $\beta_0$ such that
\begin{equation}\label{T1-1}
\|T_1f\|_2\leq C\|f\|_2
\end{equation}
holds for $0<\beta\le 1-\beta_0$.
\end{lemma}
\begin{proof}%[Proof of Lemma \ref{uni-est21}]

In view of \eqref{T11}, we denote
\begin{equation*}\label{}
K_1(x)=\frac{\Omega(x)}{|x|^{n-\beta}}\chi_\beta(|x|).
\end{equation*}
Then to prove \eqref{T1-1}, our main target is to  prove that there exists an absolute constant $C>0$ independent $\beta$ such that
\begin{equation}\label{K1-F}
\|\widehat{K_1}(y)\|_{L^\infty}\leq C, \ \ 0<\beta<1.
\end{equation}
 By cancellation of $\Omega(x)$ (see \eqref{4-9-15}), one has $\int_{\mathbb{S}^1} K_1(x) ds=0$.  Since $K_1(x)$ is supported on $|x|\le \frac 2\beta$, we have
\begin{equation*}
\widehat{K_1}(y)=\int_{\R^n} e^{2\pi ix\cdot y}K_1(x)\,dx=\int_{|x|\leq\frac{2}{\beta}} (e^{2\pi ix\cdot y}-1)K_1(x)\,dx
\end{equation*}
Since \eqref{K1-F} is a pointwise estimate, we will estimate $\widehat{K_1}(y)$ by different values of $y$.  If $|y|<\frac{\beta}{2}$,  it is direct to estimate
\begin{equation}\label{s-1}
\begin{split}
|\widehat{K_1}(y)|&\leq C|y|\int_{|x|\leq \frac 2\beta}|x|\frac{1}{|x|^{n-\beta}}\,dx
\\ &\leq C\frac{2^\beta}{\beta+1}\beta^{-\beta}.
\end{split}
\end{equation}

Then there exists an absolute constant $C>0$ such that
\begin{equation}\label{K1-F1}
|\widehat{K_1}(y)|\leq C,  0<\beta<n, \ |y|< \frac{\beta}{2}.
\end{equation}

If $\frac{\beta}{2}\leq |y|\leq\beta,$ we rewrite $\widehat{K_1}(y)$ as
\begin{equation*}\begin{split}
\widehat{K_1}(y)&=\int_{|x|<\frac{1}{|y|}} e^{2\pi ix\cdot y}K_1(x)\,dx+\int_{\frac{1}{|y|}\leq|x|\leq\frac{2}{\beta}} e^{2\pi ix\cdot y}K_1(x)\,dx
\\&=\int_{|x|<\frac{1}{|y|}} (e^{2\pi ix\cdot y}-1)K_1(x)\,dx+\int_{\frac{1}{|y|}\leq|x|\leq\frac{2}{\beta}} e^{2\pi ix\cdot y}K_1(x)\,dx
\end{split}
\end{equation*}
Similar to \eqref{s-1}, it deduces
\begin{equation*}\begin{split}
|\int_{|x|<\frac{1}{|y|}} (e^{2\pi ix\cdot y}-1)K_1(x)\,dx|\leq \frac{2^\beta}{\beta+1}\beta^{-\beta}.
\end{split}
\end{equation*}
Moreover, we have
\begin{equation*}\begin{split}
|\int_{\frac{1}{|y|}\leq|x|\leq\frac{2}{\beta}} e^{2\pi ix\cdot y}K_1(x)\,dx|\leq C\frac{2^\beta-1}{\beta}\beta^{-\beta}.
\end{split}
\end{equation*}

Consequently, there exists an absolute constant $C>0$ such that
\begin{equation}\label{K1-F2}
|\widehat{K_1}(y)|\leq C(\frac{2^\beta}{\beta+1}\beta^{-\beta}+\frac{2^\beta-1}{\beta}\beta^{-\beta})\le C, \ 0<\beta<n, \frac{\beta}{2}\le |y|\le\beta.
\end{equation}

As for $ |y|>\beta,$ $\widehat{K_1}(y)$ can be divided into
\begin{equation}\label{K1-F31}
\begin{split}
\widehat{K_1}(y)&=\int_{|x|<\frac{2}{|y|}} e^{2\pi ix\cdot y}K_1(x)\,\mathrm{d}x+\int_{\frac{2}{|y|}\leq|x|\leq\frac{2}{\beta}} e^{2\pi ix\cdot y}K_1(x)\,\mathrm{d}x
\\&=\int_{|x|<\frac{2}{|y|}} \big(e^{2\pi ix\cdot y}-1\big)K_1(x)\,\mathrm{d}x+\int_{\frac{2}{|y|}\leq|x|\leq\frac{2}{\beta}} e^{2\pi ix\cdot y}K_1(x)\,\mathrm{d}x.
\end{split}
\end{equation}
For the first term on the right hand of the above equality, we easily find that
\begin{equation}\label{K1-F311}
\begin{split}
\Big|\int_{|x|<\frac{2}{|y|}} \big(e^{2\pi ix\cdot y}-1\big)K_1(x)\,\mathrm{d}x\Big|&\leq C|y|\int_{|x|<\frac{2}{|y|}}|x|\frac{1}{|x|^{n-\beta}}\,\mathrm{d}x
\\&\leq \frac{C}{(\beta+1)|y|^\beta}\leq \frac{C}{\beta+1}\beta^{-\beta}.
\end{split}
\end{equation}
For the second term,
we  choose $z=\frac{y}{2|y|^2}$ with $|z|=\frac{1}{2|y|}<\frac{1}{2\beta}$ such that $e^{2\pi iy\cdot z}=-1$
and
\begin{equation*}
\int_{\R^n} e^{2\pi ix\cdot y}K_1(x)\,\mathrm{d}x=\frac12\int_{\R^n} e^{2\pi ix\cdot y}\big(K_1(x)-K_1(x-z)\big)\,\mathrm{d}x,
\end{equation*}
and moreover, we have
\begin{equation}\label{K1-F321}
\begin{split}
\int_{\frac{2}{|y|}\leq|x|\leq\frac{2}{\beta}} e^{2\pi ix\cdot y}K_1(x)\,\mathrm{d}x=&\frac12\int_{\frac{2}{|y|}\leq|x|\leq\frac{2}{\beta}} e^{2\pi ix\cdot y}\big(K_1(x)-K_1(x-z)\big)\,\mathrm{d}x
\\&-\frac12\int_{\frac{2}{|y|}\leq|x+z|,~ |x|\leq\frac{2}{|y|}} e^{2\pi ix\cdot y}K_1(x)\,\mathrm{d}x
\\&+\frac12\int_{|x+z|\leq\frac{1}{|y|},~ |x|\geq\frac{2}{|y|}} e^{2\pi ix\cdot y}K_1(x)\,\mathrm{d}x\\
&+\frac12\int_{|x+z|\ge\frac{2}{\beta}} e^{2\pi ix\cdot y}K_1(x) \,\mathrm{d}x
\\:= &I+J+K+L.
\end{split}
\end{equation}
To estimate the term $I$, we see that
\begin{equation}\label{F321-I}
\begin{split}
2I&=\int_{\frac{2}{|y|}\leq|x|<\frac{1}{\beta},~|x-z|\leq\frac{1}{\beta}}\Big(\frac{\Omega(x)}{|x|^{n-\beta}}-\frac{\Omega(x-z)}{|x-z|^{n-\beta}}\Big)e^{2\pi ix\cdot y}\,\mathrm{d}x
\\&+\int_{\frac{1}{\beta}\leq|x|\leq\frac{2}{\beta},~|x-z|\leq\frac{1}{\beta}}\Big(\frac{\Omega(x)}{|x|^{n-\beta}}\chi_\beta(x)-\frac{\Omega(x-z)}{|x-z|^{n-\beta}}\Big)e^{2\pi ix\cdot y}\,\mathrm{d}x
\\&+\int_{\frac{2}{|y|}\leq|x|<\frac{1}{\beta},~|x-z|\geq\frac{1}{\beta}}\Big(\frac{\Omega(x)}{|x|^{n-\beta}}-\frac{\Omega(x-z)}{|x-z|^{n-\beta}}\chi_\beta(x-z)\Big)e^{2\pi ix\cdot y}\,\mathrm{d}x
\\&+\int_{\frac{1}{\beta}\leq|x|\leq\frac{2}{\beta},~|x-z|\geq\frac{1}{\beta}}\Big(\frac{\Omega(x)}{|x|^{n-\beta}}\chi_\beta(x)-\frac{\Omega(x-z)}{|x-z|^{n-\beta}}\chi_\beta(x-z)\Big)e^{2\pi ix\cdot y}\,\mathrm{d}x
\\&:=I_1+I_2+I_3+I_4.
\end{split}
\end{equation}
We first estimate $I_2.$ Thanks to $|x-z|\geq|x|-|z|\geq\frac{1}{\beta}-\frac{1}{2|y|}\geq\frac{1}{2\beta},$ one has
\begin{equation}\label{I-2}
\begin{split}
|I_2|&\leq\int_{\frac{1}{\beta}\leq|x|\leq\frac{2}{\beta}}\frac{1}{|x|^{n-\beta}}\,\mathrm{d}x
+\int_{\frac{1}{2\beta}\leq|x-z|\leq\frac{1}{\beta}}\frac{1}{|x-z|^{n-\beta}}\,\mathrm{d}x
\\&\leq C\frac{2^\beta-1}{\beta}\beta^{-\beta}+C\frac{1-2^{-\beta}}{\beta}\beta^{-\beta}.
\end{split}
\end{equation}
Then, thanks to  $|x|=|x-z+z|\geq|x-z|-|z|\geq\frac{1}{\beta}-\frac{1}{2|y|}\geq\frac{1}{2\beta},$ $I_3$ can be estimated as follows:
\begin{equation}\label{I-3}
\begin{split}
|I_3|&\leq\int_{\frac{1}{2\beta}\leq|x|\leq\frac{1}{\beta}}\frac{1}{|x|^{n-\beta}}\,\mathrm{d}x
+\int_{\frac{1}{\beta}\leq|x-z|\leq\frac{2}{\beta}}\frac{1}{|x-z|^{n-\beta}}\,\mathrm{d}x
\\&\leq C\frac{1 -2^{-\beta}}{\beta}\beta^{-\beta}+C\frac{2^\beta-1}{\beta}\beta^{-\beta}.
\end{split}
\end{equation}
The term $I_4$ is directly estimated as
\begin{equation}\label{I-4}
\begin{split}
|I_4|&\leq\int_{\frac{1}{\beta}\leq|x|\leq\frac{2}{\beta}}\frac{1}{|x|^{n-\beta}}\,\mathrm{d}x
+\int_{\frac{1}{\beta}\leq|x-z|\leq\frac{2}{\beta}}\frac{1}{|x-z|^{n-\beta}}\,\mathrm{d}x
\\&\leq C\frac{2^\beta-1}{\beta}\beta^{-\beta}.
\end{split}
\end{equation}
Now we deal with $I_1.$  $I_1$ can be decomposed into two terms.
\begin{equation*}
\begin{split}
I_1&=\int_{4|z|=\frac{2}{|y|}\leq|x|<\frac{1}{\beta},~|x-z|\leq\frac{1}{\beta}}\Big(\frac{1}{|x|^{n-\beta}}\big(\Omega(x)-\Omega(x-z)\big) \Big)e^{2\pi ix\cdot y}\,\mathrm{d}x\\+&
\int_{4|z|=\frac{2}{|y|}\leq|x|<\frac{1}{\beta},~|x-z|\leq\frac{1}{\beta}}\Big(\Omega(x-z)\big(\frac{1}{|x|^{n-\beta}}-\frac{1}{|x-z|^{n-\beta}}\big) \Big)e^{2\pi ix\cdot y}\,\mathrm{d}x\\=&I_{11}+I_{22}.
\end{split}
\end{equation*}
Since $\Omega$ is  homogeneous of degree 0 by \eqref{4-9-15},
\begin{equation}
I_{11}=\int_{4|z|\leq|x|<\frac{1}{\beta},~|x-z|\leq\frac{1}{\beta}}\Big(\frac{1}{|x|^{n-\beta}}\big(\Omega(\frac{x}{|x|})-\Omega(\frac{x-z}{|x-z|})\big) \Big)e^{2\pi ix\cdot y}\,\mathrm{d}x.
\end{equation}
To estimate $|\frac{x}{|x|}-\frac{x-z}{|x-z|}|$, let
$$
f(x)=\frac{x}{|x|},\qquad f_i(x)=\frac{x_i}{|x|},\quad i=1,2,\ldots,n.
$$
Then
\begin{equation*}
\partial_jf_i(x)=\frac{\mathbf{\delta}_{ij}}{|x|}
-\frac{x_jx_i}{|x|^{3}},\quad i,j=1,2,\ldots,n,
\end{equation*}
where $\delta_{ij}=1$ if $i=j$ and  $\delta_{ij}=0$ if $i\neq j$.
In this case, since $|x-z|\geq|x|-|z|\geq4|z|-|z|\geq3|z|,$ it concludes that for any $0\le t\le 1$, we have $|x-tz|\ge |x-z|-(1-t)|z|\ge \frac12|x-z|$. By the mean value theorem, one has
\begin{equation*}\begin{split}
 |f_i(x-z)-f_i(x)|&=\bigg|\int_0^1 \nabla f_i(x-tz)\cdot z\, \mathrm{d}t\bigg|\\
 &\le C\bigg|\int_0^1 \frac{|z|}{|x-tz|}\, \mathrm{d}t\bigg|\\
 &\le C\frac{|z|}{|x-z|}
\end{split}
\end{equation*}
for $i=1,2,\ldots,n,$ and
\begin{equation*}\begin{split}
 |f(x-z)-f(x)|&\le \sum_{j=1}^n |f_j(x-z)-f_j(x)|\le  C\frac{|z|}{|x-z|}\leq C\frac{|z|}{|x|}
\end{split}
\end{equation*}
for some absolute constant $C$ depending on $n$.
Therefore, it follows from the definition of $\omega$ in \eqref{4-9-16} that
\begin{equation*}\begin{split}
 |\Omega(\frac{x-z}{|x-z|})-\Omega(\frac{x}{|x|})|&\le \omega(C\frac{|z|}{|x|}).
\end{split}
\end{equation*}
Consequently,
\begin{equation}\label{I-1}
\begin{split}
I_{11}\leq &C\int_{4|z|\leq|x|<\frac{1}{\beta}} \frac{1}{|x|^{n-\beta}}\omega(\frac{C|z|}{|x|})\,\mathrm{d}x\\
\\ \leq & C\int_{4|z|}^{\frac{1}{\beta}}r^{\beta-1}\omega(\frac{C|z|}{r})\,\mathrm{d}r\\ = &
C|z|^{\beta}\int_{C|z|\beta}^{\frac{C}{4}}\frac{\omega(\delta)}{\delta}\delta^{-\beta}\,\mathrm{d}\delta\\= &
C|z|^{\beta}(\int_{C|z|\beta}^{1}\frac{\omega(\delta)}{\delta}\delta^{-\beta}\,\mathrm{d}\delta
+\int_{1}^{\frac{C}{4}}\frac{\omega(\delta)}{\delta}\delta^{-\beta}\,\mathrm{d}\delta)\\ \leq &
C\beta^{-\beta}\int_0^1\frac{\omega(\delta)}{\delta}\,\mathrm{d}\delta+C|z|^{\beta}\\ \leq &C\beta^{-\beta}.
\end{split}
\end{equation}
For $I_{22},$ we need estimate $$|\frac{1}{|x|^{n-\beta}}-\frac{1}{|x-z|^{n-\beta}}|.$$
Adopting to  the  similar method to estimate $|\frac{x}{|x|}-\frac{x-z}{|x-z|}|$ as above, we get
\begin{equation*}\begin{split}
 |\frac{1}{|x|^{n-\beta}}-\frac{1}{|x-z|^{n-\beta}}|\le  C\frac{|z|}{|x-z|^{n-\beta+1}}.
\end{split}
\end{equation*}
Therefore, by using the boundness of $\Omega$ in \eqref{4-9-15}, we obtain
\begin{equation*}\begin{split}
I_{22}&\leq C|z|\int_{3|z|\leq|x-z|\leq \frac{1}{\beta}}\frac{1}{|x-z|^{n-\beta+1}}\,\mathrm{d}x
\\ \leq& C|z|\int_{3|z|}^{\frac{1}{\beta}}r^{\beta-2}\,\mathrm{d}r.
\\ \leq&C \frac{|z|^{\beta}}{1-\beta}\leq C \frac{\beta^{-\beta}}{1-\beta}
\end{split}
\end{equation*}
for $0<\beta<1$.

Collecting the estimate of $I_{11}$ and $I_{22}$
yields
\begin{equation}\label{I-E}
\begin{split}
|I|\le  \frac{C}{1-\beta}\beta^{-\beta}\le C.
\end{split}
\end{equation}
for   $0<\beta\le 1-\beta_0$ with $0<\beta_0<\frac12$, where $C$ is a constant depending on $\beta_0$. It is here that we need to impose the restriction $0<\beta\le 1-\beta_0$.

Concerning the term $J$, thanks to $|x|\geq|x+z|-|z|\geq 4|z|-|z|\geq 3|z|$, one has
\begin{equation}\label{J-E}
|J|\leq\int_{3|z|\leq|x|\leq4|z|}\frac{1}{|x|^{n-\beta}}\,\mathrm{d}x
\leq C\frac{1-(\frac34)^{-\beta}}{\beta}\beta^{-\beta}.
\end{equation}
Utilizing $|x|\leq|x+z|+|z|\leq 4|z|+|z|\leq 5|z|$, the term $K$ can be bounded by
\begin{equation}\label{K-E}
|K|\leq\int_{4|z|\leq|x|\leq5|z|}\frac{1}{|x|^{n-\beta}}\,\mathrm{d}x
\leq C\frac{1-(\frac45)^{-\beta}}{\beta}\beta^{-\beta}.
\end{equation}
As for the term $L$, the fact that $\frac{2}{\beta}\ge |x|\ge |x+z|-|z|\ge \frac{2}{\beta}-\frac{1}{2\beta}=\frac{3}{2\beta}$ enables us to conclude
\begin{equation}\label{L-E}
\begin{split}
|L| \leq\int_{\frac{3}{2\beta}\leq|x|\leq\frac{2}{\beta}}\frac{1}{|x|^{n-\beta}}\,\mathrm{d}x
 \leq\frac{1}{\beta}\left(\Big(\frac{2}{\beta}\Big)^\beta-\Big(\frac{3}{2\beta}\Big)^\beta\right)
 \leq\frac{2^\beta-(\frac{3}{2})^\beta}{\beta}\beta^{-\beta}.
\end{split}
\end{equation}
Substituting \eqref{I-E}-\eqref{L-E} into \eqref{K1-F321}, we readily obtain that there exists an absolute constant $C>0$ such that
\begin{equation}\label{K12}
\Big|\int_{\frac{2}{|y|}\leq|x|\leq\frac{2}{\beta}} e^{2\pi ix\cdot y}K_1(x)\,\mathrm{d}x\Big|\le \frac{C}{1-\beta}
\end{equation}
for $0<\beta<1$.
In view of \eqref{K1-F311}, \eqref{K12} and \eqref{K1-F31}, there exists an absolute constant $C>0$ such that
\begin{equation}\label{K1-F3}
\big|\widehat{K_1}(y)\big|\le \frac{C}{1-\beta}, \,\ \quad 0<\beta<1,\,\,\quad |y|>\beta.
\end{equation}
Combining  \eqref{K1-F1}, \eqref{K1-F2} with \eqref{K1-F3}, we finish the proof of \eqref{K1-F}. Applying \eqref{K1-F}, one has
\begin{equation*}\label{T1-1+}\begin{split}
\|T_1f\|_{L^2}=\|\widehat{K_1}\widehat{f}\|_{L^2}\le C\|\widehat{f}\|_{L^2}=\frac{C}{1-\beta}\|f\|_{L^2},
\end{split}
\end{equation*}
for $0<\beta<1$.
 Hence \eqref{T1-1} is proved and the proof of the lemma is complete.
\end{proof}

\section{Proof of Main Result}

In this section, our main task is to prove Theorem \ref{theorem21}. Then the proof of  our main result Theorem \ref{theorem2} will be resulted from Theorem \ref{theorem21} and Lemma \ref{uni-est1}.  The following lemma is crucial to prove Theorem \ref{theorem21}.
\begin{lemma}\label{MainLemma}
 Let $Q$ be a cube in $\mathbb{R}^n$ and $f\in L^2(\mathbb{R}^n)$ with ${\rm supp}  f\subset {\mathbb{R}^n}\setminus 4Q$. Suppose that  $\mathcal{M}(|T_1 f|^2)(x_0)\le a^2$ and $\mathcal{M}(f^2)(x_0)\le b^2$ for a point $x_0\in 3Q$, where $a$ and $b$ are two positive constants. Then, there exists an absolute constant $C>0$ depending on $n, B_1$ and $B_2$ but not on $\beta$, such that
 \begin{equation}\label{Main1}
 \mathcal{M}(|T_1 f|^2)(x)\le \max\{5^na^2, (a+ Cb)^2\}
\end{equation}
holds for any $x\in Q$ and  $0<\beta<\frac n2$.
\end{lemma}
\begin{proof} [Proof] It is clear that $(T_1 f)(x)$ is well defined for any $x\in 3Q$ and it is claimed that
\begin{equation}\label{4-6-1}
|T_1 f|(x)\le a+Cb, \ x\in 3Q.
\end{equation}
Here and in the following, we denote $C$ an absolute positive constant depending on $n, B_1$ and $B_2$ but not on $\beta$. In fact,
\begin{equation}\label{4-6-2}
\begin{split}
|T_1 f|(x)\le & |T_1 f|(x_0)+|T_1 f(x)-T_1 f(x_0)|\\
\le & \mathcal{M}(T_1 f)(x_0)+|T_1 f(x)-T_1 f(x_0)|.
\end{split}
\end{equation}
By H$\ddot{o}$lder inequality,
\begin{equation}\label{4-6-3}
\mathcal{M}(T_1 f)(x_0)\le \sqrt{\mathcal{M}(|T_1 f|^2)(x_0)}\le a.
\end{equation}
In view of \eqref{4-6-2} and \eqref{4-6-3}, to prove \eqref{4-6-1}, we only need to prove
\begin{equation}\label{4-6-6}
|T_1 f(x)-T_1 f(x_0)|\le Cb.
\end{equation}

 Note that
\begin{equation}\label{2}
\begin{split}
T_1 f(x)-T_1 f(x_0)=&\int_{\mathbb{R}^n\setminus 4Q}\Omega(x-y)(\frac{1}{|x-y|^{n-\beta}}-\frac{1}{|x_0-y|^{n-\beta}})\chi_\beta(|x-y|)f(y) dy\\
+&\int_{\mathbb{R}^n\setminus 4Q}\frac{\Omega(x-y)-\Omega(x_0-y)}{|x_0-y|^{n-\beta}}\chi_\beta(|x-y|)f(y) dy\\
+&\int_{\mathbb{R}^n\setminus 4Q}\frac{\Omega(x_0-y)}{|x_0-y|^{n-\beta}}(\chi_\beta(|x-y|)-\chi_\beta(|x_0-y|))f(y) dy\\
\equiv& I_1+I_2+I_3.
\end{split}
\end{equation}

We first estimate $I_1$. It holds that
\begin{equation}\label{I1-1}
\begin{split}
|I_1|=&|\int_{\mathbb{R}^n\setminus 4Q}\Omega(x-y)(\frac{1}{|x-y|^{n-\beta}}-\frac{1}{|x_0-y|^{n-\beta}})\chi_\beta(|x-y|)f(y) dy|\\
\le &C(\int_{\mathbb{R}^n\setminus 4Q}(\frac{1}{|x-y|^{n-\beta}}-\frac{1}{|x_0-y|^{n-\beta}})^{\frac{n}{n-\beta}}|f(y)|^
{\frac{n}{n-\beta}} dy)^{\frac{n-\beta}{n}}(\int_{\mathbb{R}^n\setminus 4Q} |\chi_\beta(|x-y|)|^{\frac{n}{\beta}}dy)^{\frac{\beta}{n}}\\
\le& C(\displaystyle{\sum_{m=2}^\infty}\int_{2^{m+1}Q\setminus 2^mQ}(\frac{1}{|x-y|^{n-\beta}}-\frac{1}{|x_0-y|^{n-\beta}})^{\frac{n}{n-\beta}}
|f(y)|^{\frac{n}{n-\beta}} dy)^{\frac{n-\beta}{n}} [(\frac{2}{\beta})^n]^{\frac{\beta}{n}}\\
\le& C(\displaystyle{\sum_{m=2}^\infty}\int_{2^{m+1}Q\setminus 2^mQ}(\frac{1}{|x-y|^{n-\beta}}-\frac{1}{|x_0-y|^{n-\beta}})^{\frac{n}{n-\beta}}
|f(y)|^{\frac{n}{n-\beta}} dy)^{\frac{n-\beta}{n}}.
\end{split}
\end{equation}
Then for $x\in 3Q, y\in 2^{m+1}Q\setminus 2^mQ, m=2, 3, \cdots $, by Lagrange's mean value theorem, one has
\begin{equation}\label{4-2-1}
|\frac{1}{|x-y|^{n-\beta}}-\frac{1}{|x_0-y|^{n-\beta}}|\le\displaystyle{\sup_{t\in[0,1]}}\frac{(n-\beta)|x-x_0|}{|tx_0+(1-t)x-y|^{n+1-\beta}}.
\end{equation}
Moreover, one has
\begin{equation}\label{4-2-3}
\frac{|y-\bar y|}{|z-y|}\le 4
\end{equation}
for any $z\in 3Q,  y\in 2^{m+1}Q\setminus 2^mQ, m=2, 3, \cdots$, where $\bar y$ is the center of $Q$, and
\begin{equation}\label{4-2-3+}
|x-x_0|\le 3\sqrt{n}l,
\end{equation}
for $x, x_0\in 3Q$, where $l$ denotes the side-length of $Q$. It follows from \eqref{4-2-1}, \eqref{4-2-3} and \eqref{4-2-3+} that
\begin{equation}\label{4-2-6}
\begin{split}
|\frac{1}{|x-y|^{n-\beta}}-\frac{1}{|x_0-y|^{n-\beta}}|^{\frac{n}{n-\beta}}\le & [\frac{3\cdot 4^{n+1-\beta}n\sqrt{n}l}{|y-\bar y|^{n+1-\beta}}]^{\frac{n}{n-\beta}}\le [\frac{3\cdot 4^{n+1-\beta}n\sqrt{n}l}{(2^ml)^{n+1-\beta}}]^{\frac{n}{n-\beta}}\\
 = & \frac{(3\cdot 4^{n+1-\beta}n\sqrt{n}l)^{\frac{n}{n-\beta}}}{(2^ml)^{n}(2^m)^{\frac{n}{n-\beta}}}\le \frac{C2^{-m}}{|2^{m+1}Q|}.
\end{split}
\end{equation}
Thanks to \eqref{4-2-6}, we obtain
\begin{equation}\label{I1-2}
\begin{split}
&\displaystyle{\sum_{m=2}^\infty}\int_{2^{m+1}Q\setminus 2^mQ}|\frac{1}{|x-y|^{n-\beta}}-\frac{1}{|x_0-y|^{n-\beta}}|^{\frac{n}{n-\beta}}
|f(y)|^{\frac{n}{n-\beta}} dy\\
\le &\displaystyle{\sum_{m=2}^\infty} \frac{C2^{-m}}{|2^{m+1}Q|}\int_{2^{m+1}Q} |f(y)|^{\frac{n}{n-\beta}} dy\\
\le &C(\mathcal{M}(f^2)(x_0))^{\frac{n}{2(n-\beta)}}.
\end{split}
\end{equation}
Substituting \eqref{I1-2} into \eqref{I1-1} arrives at
\begin{equation}\label{I1-3}
\begin{split}
|I_1|\le C(\mathcal{M}(f^2)(x_0))^{\frac{1}{2}}\le Cb.
\end{split}
\end{equation}
Next we estimate $I_2$. Note that
\begin{equation}\label{I2-1}
\begin{split}
|I_2|=&|\int_{\mathbb{R}^n\setminus 4Q}\frac{\Omega(x-y)-\Omega(x_0-y)}{|x_0-y|^{n-\beta}}\chi_\beta(|x-y|)f(y) dy|\\
\le & (\int_{\mathbb{R}^n\setminus 4Q}\frac{|\Omega(x-y)-\Omega(x_0-y)|^{\frac{n}{n-\beta}}}{|x_0-y|^{n}}|f(y)|^{\frac{n}{n-\beta}} dy)^{\frac{n-\beta}{n}}(\int_{\mathbb{R}^n\setminus 4Q} |\chi_\beta(|x-y|)|^{\frac{n}{\beta}}dy)^{\frac{\beta}{n}}\\
\le & C(\int_{\mathbb{R}^n\setminus 4Q}\frac{|\Omega(x-y)-\Omega(x_0-y)|}{|x_0-y|^{n}}|f(y)|^{\frac{n}{n-\beta}} dy)^{\frac{n-\beta}{n}}[(\frac{2}{\beta})^n]^{\frac{\beta}{n}}\\
\le & C(\int_{\mathbb{R}^n\setminus 4Q}\frac{|\Omega(x-y)-\Omega(x_0-y)|}{|x_0-y|^{n}}|f(y)|^{\frac{n}{n-\beta}} dy)^{\frac{n-\beta}{n}}.
\end{split}
\end{equation}
Similar to \eqref{4-2-1}, for $x\in 3Q, y\in 2^{m+1}Q\setminus 2^mQ, m=2, 3, \cdots $, by Lagrange's mean value theorem, one has
\begin{equation}\label{4-2-2}
|\frac{x-y}{|x-y|}-\frac{x_0-y}{|x_0-y|}|\le\displaystyle{\sup_{t\in[0,1]}}\frac{2|x-x_0|}{|tx_0+(1-t)x-y|^{n+1}},
\end{equation}
Applying \eqref{4-2-3} and \eqref{4-2-3+} yields
\begin{equation}\label{4-2-7}
\begin{split}
&|\Omega(x-y)-\Omega(x_0-y)|=|\Omega(\frac{x-y}{|x-y|})-\Omega(\frac{x_0-y}{|x_0-y|})|\\
\le &\omega(|\frac{x-y}{|x-y|}-\frac{x_0-y}{|x_0-y|}|)\le \omega(\frac{24\sqrt{n}l}{|y-\bar y|})\le \omega(\frac{24\sqrt{n}}{2^{m-1}}).
\end{split}
\end{equation}
Thanks to \eqref{4-2-7}, one has, for $0<\beta<\frac n2$,
\begin{equation}\label{I2-2}
\begin{split}
&\int_{\mathbb{R}^n\setminus 4Q}\frac{|\Omega(x-y)-\Omega(x_0-y)|}{|x_0-y|^{n}}|f(y)|^{\frac{n-\beta}{n}} dy\\
=&\displaystyle{\sum_{m=2}^\infty}\int_{2^{m+1}Q\setminus 2^mQ} \frac{|\Omega(x-y)-\Omega(x_0-y)|}{|x_0-y|^{n}}|f(y)|^{\frac{n-\beta}{n}} dy\\
&\le \displaystyle{\sum_{m=2}^\infty} \omega(\frac{24\sqrt{n}}{2^{m-1}})\int_{2^{m+1}Q\setminus 2^mQ}\frac{4^n|f(y)|^{\frac{n-\beta}{n}}}{|y-\bar y|^n} dy\\
&\le \displaystyle{\sum_{m=2}^\infty} \frac{2^{3n}\omega(\frac{24\sqrt{n}}{2^{m-1}})}{|2^{m+1}Q|}\int_{2^{m+1}Q} |f(y)|^{\frac{n-\beta}{n}} dy\\
&\le C(\mathcal{M}(f^2)(x_0))^{\frac{n}{2(n-\beta)}}\displaystyle{\sum_{m=2}^\infty} \omega(\frac{24\sqrt{n}}{2^{m-1}}).
\end{split}
\end{equation}
For any fixed $K>0, 0<\tau<1$, direct calculations give
\begin{equation}\label{I2-3}
\begin{split}
\ln\frac{1}{K\tau}\displaystyle{\sum_{m=2}^\infty} \omega(K\tau^{m-1})\le \displaystyle{\sum_{m=2}^\infty} \int_{K\tau^m}^{K\tau^{m-1}}\frac{\omega(t)}{t} dt=\int_0^{K\tau} \frac{\omega(t)}{t} dt.
\end{split}
\end{equation}
Taking $K=24\sqrt{n}, \tau=\frac12$ in \eqref{I2-3}, one has
\begin{equation}\label{I2-6}
\begin{split}
\displaystyle{\sum_{m=2}^\infty} \omega(\frac{24\sqrt{n}}{2^{m-1}})\le C\int_0^{12\sqrt{n}} \frac{\omega(t)}{t} dt\le C(\int_0^1 \frac{\omega(t)}{t} dt+\int_1^{12\sqrt{n}} \frac{\omega(t)}{t} dt)\le C.
\end{split}
\end{equation}
It follows from \eqref{I2-1},\eqref{I2-2} and \eqref{I2-6} that, for $0<\beta<\frac n2$,
\begin{equation}\label{I2-7}
\begin{split}
|I_2|\le C(\mathcal{M}(f^2)(x_0))^{\frac{1}{2}}\le Cb.
\end{split}
\end{equation}
Now we estimate $I_3$. Note that
\begin{equation}\label{I3-1}
\begin{split}
|I_3|= &\int_{\mathbb{R}^n\setminus 4Q}\frac{\Omega(x_0-y)}{|x_0-y|^{n-\beta}}(\chi_\beta(|x-y|)-\chi_\beta(|x_0-y|))f(y) dy\\
= &|\displaystyle{\sum_{m=2}^\infty}\int_{2^{m+1}Q\setminus 2^mQ}\frac{\Omega(x_0-y)}{|x_0-y|^{n-\beta}})(\chi_\beta(|x-y|)-\chi_\beta(|x_0-y|))f(y) dy|\\
\le &\displaystyle{\sum_{m=2}^\infty}(\int_{2^{m+1}Q\setminus 2^mQ}\frac{1}{|x_0-y|^{n}}|f(y)|^{\frac{n}{n-\beta}} dy)^{\frac{n-\beta}{n}}\\
&\cdot(\int_{2^{m+1}Q\setminus 2^mQ}|\chi_\beta(|x-y|)-\chi_\beta(|x_0-y|)|^{\frac{n}{\beta}}dy)^{\frac{\beta}{n}}.
\end{split}
\end{equation}
Since
$$
2^ml-3l\le |y-\bar y|-|x_0-\bar y|\le |x_0-y|\le |y-\bar y|+|x_0-\bar y|\le 2^{m+1}l+3l,
$$
one has $\frac{1}{\beta}\sim |x_0-y|\sim 2^ml$ and similarly $\frac{1}{\beta}\sim |x-y|\sim 2^ml$, and $\beta l\sim \frac{1}{2^m}$. Therefore, for $0<\beta<\frac{n}{2}$, it yields
\begin{equation}\label{I3-2}
\begin{split}
|I_3|\le C &\displaystyle{\sum_{m=2}^\infty}(\frac{1}{(2^ml)^n} \int_{2^{m+1}Q\setminus 2^mQ}|f(y)|^{\frac{n}{n-\beta}} dy)^{\frac{n-\beta}{n}}\cdot (\int_{2^{m+1}Q\setminus 2^mQ}|\chi_\beta(|x-y|)-\chi_\beta(|x_0-y|)|^{\frac{n}{\beta}}dy)^{\frac{\beta}{n}}\\
\le &  C\displaystyle{\sum_{m=2}^\infty}(\frac{2^n}{(2^{m+1}l)^n} \int_{2^{m+1}Q}|f(y)|^{\frac{n}{n-\beta}} dy)^{\frac{n-\beta}{n}}\\
&\cdot (\int_{2^{m+1}Q\setminus 2^mQ}|\chi'_\beta(t_0|x-y|+(1-t_0)|x_0-y|)|^{\frac{n}{\beta}}|x-x_0|^{\frac{n}{\beta}} dy)^{\frac{\beta}{n}}\\
\le & C \displaystyle{\sum_{m=2}^\infty}(\frac{1}{(2^{m+1}l)^n} \int_{2^{m+1}Q}|f(y)|^2 dy)^{\frac{1}{2}}\cdot (\int_{2^{m+1}Q\setminus 2^mQ}(\beta l)^{\frac{n}{\beta}} dy)^{\frac{\beta}{n}}\\
\le & C(\mathcal{M}(f^2)(x_0))^{\frac{1}{2}}\displaystyle{\sum_{m=2}^\infty} (\int_{2^{m+1}Q\setminus 2^mQ}(\beta l)^{\frac{n}{\beta}} dy)^{\frac{\beta}{n}}\\
\le & C(\mathcal{M}(f^2)(x_0))^{\frac{1}{2}}\displaystyle{\sum_{m=2}^\infty} \frac{1}{2^m}(2^ml)^\beta\\
\le & C(\mathcal{M}(f^2)(x_0))^{\frac{1}{2}}\displaystyle{\sum_{m=2}^\infty} \frac{1}{2^m}(\frac{1}{\beta})^\beta\\
\le & C(\mathcal{M}(f^2)(x_0))^{\frac{1}{2}}\le Cb,
\end{split}
\end{equation}
where $t_0\in (0,1)$ in the second inequality.

Combining \eqref{2} with \eqref{I1-3}, \eqref{I2-7} and \eqref{I3-2}, we finish the proof of \eqref{4-6-6} and hence the claim \eqref{4-6-1} holds true.

Now we prove \eqref{Main1}. Fix $x\in Q$ and let $\hat{Q} $ be any cube containing $x$. If $\hat{Q}\subseteq 3Q$, then it follows from \eqref{4-6-1} that
\begin{equation}\label{4-6-9}
\begin{split}
\frac{1}{\hat{Q}}\int_{\hat{Q}} |T_1 f|^2(y) dy\le (a+Cb)^2.
\end{split}
\end{equation}
If $\hat{Q}\nsubseteq 3Q$, then $x_0 \in 5\hat{Q}$ and in view of $\mathcal{M}(|T_1 f|^2)(x_0)\le a^2$, we have
\begin{equation}\label{4-6-10}
\begin{split}
\frac{1}{\hat{Q}}\int_{\hat{Q}} |T_1 f|^2(y) dy\le \frac{5^n}{5\hat{Q}}\int_{5\hat{Q}} |T_1 f|^2(y) dy\le 5^na^2.
\end{split}
\end{equation}
Hence \eqref{Main1} follows from \eqref{4-6-9} and \eqref{4-6-10}. The proof of the lemma is complete.
\end{proof}

In the following lemma, the constants $a$ and $b$ in Lemma \ref{MainLemma} will be fixed as $a=1$ and $b=2$ respectively, and we denote
\begin{equation}\label{N}
\begin{split}
\frac{N^2}{4}=\max\{2\cdot 5^n, (2+C)^2\},
\end{split}
\end{equation}
 which is the absolute constant appeared in \eqref{Main1} with $a=1$ and $b=2$ accordingly. Then, it holds
\begin{lemma}\label{MainLemma2}
 Let $\tilde{Q}$ be a cube in $\mathbb{R}^n$ and $Q$ be one of its dyadic cubes. Assume that $f\in L^2(\mathbb{R}^n)$. Then for any $\mu>0$, we can choose a $0<\delta\le 1$ depending on $\mu$ such that
 \begin{equation}\label{Main2}
 |\{x\in Q: \mathcal{M}(|T_1 f|^2)(x) >N^2\}|\le \mu |Q|
\end{equation}
for $0<\beta< 1$, if
\begin{equation}\label{Main3}
 |\{x\in \tilde{Q}: \mathcal{M}(|T_1 f|^2)(x)\le 1\}\cap \{x\in \tilde{Q}: \mathcal{M}(f^2)(x)\le \delta^2\}|>\frac{1}{2}|\tilde{Q}|.
\end{equation}
\end{lemma}
The proof of Lemma \ref{MainLemma2} can be seen in \cite{LW} and we give a sketch of proof as follows.

\begin{proof}[Proof]
Define
\begin{equation*}\label{}
f_1(x)=
\left\{
  \begin{array}{l}
    f(x), x\in 4Q,\\
    0, x\in \mathbb{R}^n\setminus 4Q,
  \end{array}
  \right.
\end{equation*}
and $f_2(x)=f(x)-f_1(x)$. It follows from \eqref{Main3} that there exists $x\in \tilde{Q}$ such that
$$
\mathcal{M}(f_1^2)(x)\le \mathcal{M}(f^2)(x)\le \delta^2.
$$
Consequently,
\begin{equation}\label{4-9-1}
\|f_1\|^2_{L^2(\mathbb{R}^n)}=\int_{4Q} |f_1|^2(y) dy\le \delta^2|4Q|.
\end{equation}
Then, using the weak $(1,1)$ type estimate of the maximal function in Lemma \ref{MF} and the strong $(2,2)$ type estimate of $T_1$ in Lemma  \ref{uni-est21}, we deduce
\begin{equation}\label{4-19-1}
\begin{split}
 |\{x\in \tilde{Q}: \mathcal{M}(|T_1 f_1|^2)(x)>1\}|\le & |\{x\in \mathbb{R}^n: \mathcal{M}(|T_1 f_1|^2)(x)>1\}|\\
 \le & C\|T_1 f_1\|_2^2\le C\|f_1\|_2^2\le C\delta^2|4Q|,
\end{split}
\end{equation}
and hence we can choose $0<\delta\le 1$ small enough such that
\begin{equation}\label{4-9-3}
|\{x\in \tilde{Q}: \mathcal{M}(|T_1 f|^2)(x)>1\}|<\frac{1}{2}|\tilde{Q}|
\end{equation}
for $0<\beta<1$. It follows from \eqref{Main3} and \eqref{4-9-3} that
\begin{equation}\label{4-9-6}
 |\{x\in \tilde{Q}: \mathcal{M}(|T_1 f|^2)(x)\le 1\}\cap \{x\in \tilde{Q}: \mathcal{M}(f^2)(x)\le \delta^2\}\cap \{x\in \tilde{Q}: \mathcal{M}(|T_1 f_1|^2)(x)\le 1\}|>0,
\end{equation}
which implies that there exists $x_0\in \tilde{Q}$ such that
\begin{equation}\label{4-9-7}
 \mathcal{M}(|T_1 f|^2)(x_0)\le 1,  \mathcal{M}(f^2)(x_0)\le \delta^2,  \mathcal{M}(|T_1 f_1|^2)(x_0)\le 1.
\end{equation}
Then one has
\begin{equation}\label{4-9-8}
  \mathcal{M}(f_2^2)(x_0)\le \mathcal{M}(f^2)(x_0)\le \delta^2.
\end{equation}
Moreover, since
$$
|T_1 f_2|^2=|T_1 f-T_1 f_1|^2\le 2(|T_1 f|^2+|T_1 f_1|^2),
$$
it yields
\begin{equation}\label{4-9-9}
  \mathcal{M}(|T_1 f_2|^2)(x_0)\le 2(\mathcal{M}(|T_1 f|^2)(x_0)+\mathcal{M}(|T_1 f_1|^2)(x_0))\le 4,
\end{equation}
where $x_0\in \tilde{Q}\subset 3Q$ is same as in \eqref{4-9-7}. Then due to Lemma \ref{MainLemma} and  \eqref{N}, one has
 \begin{equation}\label{4-9-10}
  \mathcal{M}(|T_1 f_2|^2)(x)\le \max\{2\cdot 5^n, (2+C)^2\}=\frac{N^2}{4}
\end{equation}
for any $x\in Q$.  Similar to \eqref{4-9-9}, one can show that
\begin{equation}\label{4-9-11}
  \mathcal{M}(|T_1 f|^2)(x)\le 2(\mathcal{M}(|T_1 f_1|^2)(x)+\mathcal{M}(|T_1 f_2|^2)(x))
\end{equation}
for any $x\in \mathbb{R}^n$, which implies that
\begin{equation*}\label{}
\begin{split}
 |\{x\in Q: \mathcal{M}(|T_1 f|^2)(x)>N^2\}|\le & |\{x\in Q: \mathcal{M}(|T_1 f_1|^2)(x)>\frac{N^2}{4}\}|\\
 + & |\{x\in Q: \mathcal{M}(|T_1 f_2|^2)(x)>\frac{N^2}{4}\}|\\
 =& |\{x\in Q: \mathcal{M}(|T_1 f_1|^2)(x)>\frac{N^2}{4}\}|\\
 \le &C_1\frac{4\|T_1 f_1\|_2^2}{N^2}\le \frac{C\|f_1\|_2^2}{N^2}\le  \frac{C\delta^2|4Q|}{N^2},
\end{split}
\end{equation*}
where the weak  (1,1) type estimate of $\mathcal{M}$ in Lemma \ref{MF} and \eqref{4-9-1} have been applied. Let $0<\delta\le 1$ be small enough and then the proof of the lemma is finished.
\end{proof}

With help of Calder$\acute{\rm o}$n-Zygmund decomposition, one can prove
\begin{lemma}\label{MainLemma3}
 Let $Q$ be a cube in $\mathbb{R}^n$, $\mu$ and $\delta$ be given by Lemma \ref{MainLemma2}. Assume that $f\in L^2(\mathbb{R}^n)$ satisfying
 \begin{equation}\label{Main4}
 |\{x\in Q: \mathcal{M}(|T_1 f|^2)(x) >N^2\}|\le \mu |Q|.
\end{equation}
Then
\begin{equation}\label{Main5}
\begin{split}
  &|\{x\in Q: \mathcal{M}(|T_1 f|^2)(x) >N^2\}|\\
 \le & 2\mu(|\{x\in Q: \mathcal{M}(|T_1 f|^2)(x)> 1\}|+|\{x\in Q: \mathcal{M}(f^2)(x)> \delta^2\}|).
 \end{split}
\end{equation}
\end{lemma}
The proof of Lemma \ref{MainLemma3} is referred to \cite{LW} and we omit it here. The following lemma  follows directly from Lemma \ref{MainLemma3}.
\begin{lemma}\label{MainLemma5}
 Assume that $f\in L^2(\mathbb{R}^n)$. Then for any $\lambda>0$ and $0<\beta<1$,
 \begin{equation}\label{Main6}
\begin{split}
  &|\{x\in \mathbb{R}^n: \mathcal{M}(|T_1 f|^2)(x) >\lambda N^2\}|\\
 \le & 2\mu(|\{x\in Q: \mathcal{M}(|T_1 f|^2)(x)> \lambda\}|+|\{x\in Q: \mathcal{M}(f^2)(x)> \lambda\delta^2\}|).
 \end{split}
\end{equation}
\end{lemma}
\begin{proof}[Proof]
By replacing $f$ by $\frac{f}{\lambda}$, it is sufficient to prove \eqref{Main6} with $\lambda=1$. Let
$$
\mathbb{R}^n=\displaystyle{\bigcup_{i=1}^\infty}\bar{Q_i},
$$
where $Q_i (i=1,2,\cdots)$ are cubes with the same side-length such that $Q_i\cap Q_j=\emptyset$ if $i\neq j$ and $\bar{Q_i}(i=1,2,\cdots)$ are the closures respectively. By Lemma \ref{MainLemma3}, it is only needed to prove
\begin{equation}\label{Main7}
 |\{x\in Q_i: \mathcal{M}(|T_1 f|^2)(x) >N^2\}|\le \mu |Q_i|
\end{equation}
for any $i=1,2,\cdots$. Since
\begin{equation*}\label{Main8}
 |\{x\in \mathbb{R}^n: \mathcal{M}(|T_1 f|^2)(x) >N^2\}|\le C\frac{\|T_1 f\|_2}{N^2}\le C\frac{\|f\|_2}{N^2},
\end{equation*}
we can choose $Q_i$ large enough such that \eqref{Main7} holds for any $i=1,2,\cdots$ and the proof of the lemma is finished.
\end{proof}

Now we are ready to prove our main results.
\begin{proof}[Proof of Theorem \ref{theorem21}]

In case of  $q=2$, it is proved in Lemma \ref{uni-est21}. In case of $2<q<\infty$, due to \eqref{Main6}, one has
\begin{equation}\label{4-13-1}
\begin{split}
 \int_{\mathbb{R}^n}(\mathcal{M}(|T_1 f|^2)(x))^{\frac{q}{2}} dx=& \frac{q}{2} \int_0^\infty t^{\frac{q}{2}-1}|\{x\in\mathbb{R}^n: \mathcal{M}(|T_1 f|^2)(x)>t\}| dt\\
 =& \frac q2N^q\int_0^\infty t^{\frac{q}{2}-1}|\{x\in\mathbb{R}^n: \mathcal{M}(|T_1 f|^2)(x)>N^2t\}| dt\\
 \le & \mu qN^q\int_0^\infty t^{\frac{q}{2}-1}(|\{x\in\mathbb{R}^n: \mathcal{M}(|T_1 f|^2)(x)>t\}| \\
 &+|\{x\in\mathbb{R}^n: \mathcal{M}(f^2)(x)>t\delta^2\}|) dt.
 \end{split}
\end{equation}
Let $\mu=\frac{1}{2qN^q}$. It follows that
\begin{equation}\label{4-13-2}
\begin{split}
 \int_{\mathbb{R}^n}(\mathcal{M}(|T_1 f|^2)(x))^{\frac{q}{2}} dx
 \le & \int_0^\infty t^{\frac{q}{2}-1}|\{x\in\mathbb{R}^n: \mathcal{M}(f^2)(x)>t\delta^2\}| dt\\
 \le & \frac{1}{\delta^q}\int_0^\infty t^{\frac{q}{2}-1}|\{x\in\mathbb{R}^n: \mathcal{M}(f^2)(x)>t\}| dt\\
 \le & \frac{1}{\delta^q} (C_{\frac q2})^{\frac q2} \int_\mathbb{R}^n |f|^q dx,
 \end{split}
\end{equation}
where in the last inequality we used the strong  $(\frac q2, \frac q2)$ type inequality of the maximal function in Lemma \ref{MF} for $2<q<\infty$. Since $|T_\varepsilon f|^2(x)\le \mathcal{M}(|T_1 f|^2)(x)$ a.e., \eqref{4-11-1} in Theorem \ref{theorem3} follows.

The case $1<q<2$ can be proved by using the duality method. It is referred to \cite{stein, YJ} for more details and we  omit it here. The proof of the theorem is complete.
\end{proof}

\begin{proof}[Proof of Theorem \ref{theorem2}]
Combining Lemma \ref{uni-est1} with Theorem \ref{theorem21}, we obtain Theorem \ref{theorem2} and the proof is complete.
\end{proof}

{\bf Acknowledgements.}
Q. Jiu was partially supported by the National Natural Science Foundation of
China (NNSFC) (No. 11671273, No.11931010), key research project of the Academy for Multidisciplinary Studies of CNU, and Beijing Natural Science Foundation (BNSF) (No. 1192001).
D. Li was partially supported by the National Natural Science Foundation of
China (NNSFC) (No. 11671316).
H. Yu was partially supported
by the National Natural Science Foundation of China (NNSFC) (No. 11901040), Beijing Natural Science Foundation (BNSF) (No. 1204030)
and   Beijing Municipal Education Commission (KM202011232020).

\end{document}